\documentclass[11pt]{article}
 \usepackage{times}
 \usepackage{mathdots}
 \usepackage{color}
 \usepackage{amssymb,amscd,amsthm,latexsym}
 \usepackage{amsmath}
 \usepackage{graphicx}
 \graphicspath{ {graphs/} }
 \usepackage{wrapfig}
 \usepackage{float}
 \usepackage{amssymb}
 \usepackage{epstopdf}
 \usepackage{verbatim}
 \newtheorem{theorem}{Theorem}[section]
 
 \newtheorem{lemma}[theorem]{Lemma}

 \theoremstyle{definition}
 \newtheorem{definition}[theorem]{Definition}
 \newtheorem{example}[theorem]{Example}
 
 \newtheorem{remark}[theorem]{Remark}
 
 \newtheorem{openquestion}[theorem]{Open Question}
 
 \newtheorem{cor}[theorem]{Corollary}

 \newtheorem{rem}[theorem]{Remark}

 \def\a{\alpha~}
 
 \def\l{\lambda}
 \def\n{n \times n~}

 \def\r{\hat{\rho}}
 \def\am{algebraic multiplicity~}
 \def\gm{geometric multiplicity}
 \title{\vskip70mm Some properties of ergodicity coefficients with applications in spectral graph theory}
 \author{Rachid Marsli \footnote{ Supported by grant no. SB181014
 		of King Fahd University of Petroleum and Minerals.}
 \\ Mathematics and Statistics Department\\
 King Fahd University of Petroleum and Minerals\\Dhahran, 31261\\
 Kingdom of Saudi Arabia\\ rmarsliz@kfupm.edu.sa\\
 Frank J. Hall \\ Department of Mathematics and Statistics\\ Georgia State University\\ Atlanta, GA 30303, USA\\ fhall@gsu.edu}
 \date{\today}
\begin{document}
 	%
 	%
 	%
 	\maketitle
 	\begin{abstract}
 	The main result is Corollary \ref{cc2} which provides upper bounds on, and even better, approximates the largest non-trivial eigenvalue in absolute value of real constant row-sum matrices by the use of vector norm based ergodicity coefficients $\{\tau_p\}$. If the constant row-sum matrix is nonsingular, then it is also shown how its smallest non-trivial eigenvalue in absolute value can be bounded by using $\{\tau_p\}$.
 	In the last section, these two results are applied to bound the spectral radius of the Laplacian matrix as well as the algebraic connectivity of its associated graph.
 	Many other results are obtained. In particular,
 	Theorem \ref{t5} is a convergence theorem for $\tau_p$ and Theorem \ref{t10} says that $\tau_1$ is less than or equal to $\tau_{\infty}$ for the Laplacian matrix of every simple graph.
 	 Other discussions, open questions and examples are provided.
  	\end{abstract}
 	\noindent
 	{\it AMS Subj. Class.:} 15A18; 15B51 
 	   \begin{section}{Introduction}
 		Let $A$ be a matrix in $M_n(\mathbb{R})$, the set of $\n$ real matrices and let $v$ be a vector in $\mathbb{C}^n$.
 		Let $||v||_p$ be the $l_p$-norm of the vector $v$ with $p\in \mathbb{N} \cup \{\infty\}$.
 		A class of functions mapping $M_n(\mathbb{R})$ to $R^+ \cup\{0\}$ arises from the above vector-norms as follows:
 		\begin{equation}\label{F22}
 		\tau_p(A) = \underset{\underset{\underset{||x||_p =1}{x^T e =0}}{x\in \mathbb{R}^n}}{\max} ||A^T x||_p,
 		\end{equation}
 		where $e$ is the all $1$'s vector of $n$ components.      
 		These functions turn out to be very useful in the study of the asymptotic behavior of Markov chains \cite{Se}.
 		They form a class of the so-called ergodicity coefficients of stochastic matrices and have been the subject of interesting research. Some of the most   
        important results obtained about them are collected in
 		a comprehensive and informative 2011 work \cite{IS} by I. Ipsen and T. Selee. 
 		Stochastic matrices are emphasized in Sections 1 - 5 of this survey \cite{IS}, while extensions of these functions to general rectangular real and complex 
 		matrices are considered in Sections 6 and 7. We first summarize a few of the many important properties of these functions as given in \cite{IS}. 
                For the first theorem, see \cite[Theorem 5.1]{IS} and the references therein.
 		
 		\begin{theorem} \label{tt1}
 		Let $S, S_1$ and $S_2$ be $n\times n$ stochastic matrices, and let $||S||_p$ be the $p$-matrix norm of $S$.  
 			\begin{enumerate}
 				\item $0 \leq \tau_p(S) \leq ||S^T||_p$,\\
 				\item $|\tau_p(S_1) - \tau_p(S_2)| \leq \tau_p (S_1 - S_2)$,\\
 				\item $\tau_p(S) = 0$ if and only if $rank(S) = 1$,
 				\item $ \tau_p(S_1\,S_2) \leq \tau_p(S_1)\tau_p(S_2),$
 				\item If $S$ is irreducible and $1$ is the only eigenvalue of modulus $1$, then $|\l| \leq \tau_p (S)$ for all eigenvalues $\l \neq 1$  
                                      of $S$. 
 			\end{enumerate}       
 		\end{theorem}
 	
 		The last statement in Theorem \ref{tt1} says that when $S$ is restricted as indicated, an important application of the function $\tau_p$  
 		is that it provides an upper-bound for the absolute value of the subdominant eigenvalue of the stochastic matrix.  When $p = 1$ or 
 		$p = \infty$, the result holds for all stochastic matrices.
 		
 		\begin{theorem} \label{tt2}
 			Let $S$ be a stochastic matrix and suppose that $\l \neq 1$ is an eigenvalue of $S$. Then
 			\begin{equation}\label{F27}
			|\l| \leq \tau_1(S)
 			\end{equation}
 			and
 			\begin{equation}\label{F28}
 			|\l| \leq \tau_{\infty}(S)
 			\end{equation} 			 		 
 		\end{theorem}
 		The bound given by (\ref{F28}) is contained in \cite[Theorem 4.3]{IS}, where the authors attribute it to E. Senata \cite{Se1};
 		for (\ref{F27}) see \cite[Theorem 3.6]{IS} and the included references.\\
 		
 		Explicit forms are known for $\tau_1(S)$ and $\tau_{\infty}(S)$  with $S$ being stochastic. For $\tau_1(S)$, see \cite{Se1}.
 		\begin{theorem}\label{tt3}
 		Let $S=[s_{ij}]$ be an $\n$ stochastic matrix. Then
 		\begin{equation}\label{F26a}
 		\tau_1(S) = \frac{1}{2} \underset{i,j}{\max}\sum_{k=1}^n |s_{ik} - s_{jk}|
 		= 1- \underset{i,j}{\min}\sum_{k=1}^n \min\{s_{ik} , s_{jk}\}.
 		\end{equation} 			
 		\end{theorem} 		
 		For an explicit formula of $\tau_{\infty}$, we need the following definition.
 		
 		\begin{definition}\cite[Definition 2.3]{HM4} \label{d1}
 			Let $A= [a_{ij}]$ be an $\n$ real matrix.\\
 			Let $b_{1j}\geq \dots \geq b_{nj}$ be an arrangement in non-increasing order of 
 			$a_{1j}, \dots, a_{nj}$.
 			Then, for $j=1, \dots, n,~$ define $cs_j(A)$ in the following manner:\\
 			\hskip10mm $(1) ~~~ cs_j(A) =
 			\big(b_{1j}+\dots+b_{\frac{n-1}{2},j}\big)-\big(b_{\frac{n+3}{2},j}+\dots+b_{nj}\big),~~$
 			if $n$ is odd.\\
 			\hskip10mm $(2) ~~~ cs_j(A) = 
 			\big(b_{1j}+\dots+b_{\frac{n}{2},j}\big)-\big(b_{1+\frac{n}{2},j}+\dots+b_{nj}\big),~~$
 			if $n$ is even.\\
 			We define $\r(A)$ by:
 			$$ \r(A) = \underset{1\leq j \leq n}{\text{max}}\{cs_j(A)\}.$$
 		\end{definition}
 		
 		\begin{rem}\label{r2}
 			This definition of $cs_j(A)$ is equivalent to the one given in \cite[Lemma 3.3]{IS},
 			where, $cs_j(A)$ is denoted $\phi(A e_j)$, with $e_j$  the $j^{th}$ canonical vector. 
 		\end{rem}
 		
 		\begin{theorem} \cite[Theorem 4.2]{IS} \label{tt4}
 			Let $S	$ be an $\n$ stochastic matrix. Then
 			\begin{equation}\label{F26b}
 			\tau_{\infty}(S) = \r(S).	
 			\end{equation} 			 				
 		\end{theorem} 		
 		
 		As mentioned in \cite{IS}, this important result goes back to C.P. Tan \cite{Tan}.
 		
 		One of the topics of this paper is to show that the set of real constant row-sum matrices provides  a better theoretical framework for the study of 
                $\{\tau_p  ~|~ p \in \mathbb{N}\cap \{\infty\}\}$ compared to the set of stochastic matrices. 
	 		In the sequel and for convenience,
	 		real $\n$ constant row-sum matrices are simply called e-matrices; we shall denote the set of all such matrices by $E_n$.
 		If $A$ is an e-matrix, then it has an eigenvalue $\lambda_A$ equal to its constant row-sum. We call such eigenvalue, the trivial eigenvalue of 
 		$A$, and we call every eigenvalue of $A$ distinct from $\l_A$ a non-trivial eigenvalue of $A$.
 		The all 1's vector with $n$ components, which we denote $e$, is an eigenvector of $A$ associated with its trivial eigenvalue $\lambda_A$. 
 		The $n \times n$ all 1's matrix is denoted $J$.
 		There are many reasons behind our thinking that it is better to study the functions $\tau_p$ within $E_n$.
 		
 		\begin{enumerate}
 			\item First, these functions are well defined for all real matrices including e-matrices.
 			\item As we shall explain, it is not difficult to see that the functions $\tau_p$  are semi-norms on $E_n$,
 			 which is a  vector subspace of $M_n(\mathbb{R})$, closed under matrix multiplication.
 			\item Besides stochastic matrices, there are other important types of e-matrices such as Laplacian and circulant matrices.
 			\item Important properties of $\{\tau_p\}$ such as those given by Theorems \ref{tt2}, \ref{tt3} and \ref{tt4} can be extended over $E_n$.
 			\item The set of nonsingular matrices in $E_n$ is closed under matrix inversion. This important property of $E_n$ is explored in the last two sections of this work.
 		\end{enumerate}
 	
 			In  our paper \cite{HM4}, published in $2019$, without being aware of the explicit form of $\tau_{\infty}$ (Theorem \ref{tt4}), we obtained in a 
 			totally different manner the same result implied by the combination of Theorems \ref{tt2} and \ref{tt4},
 			that if  $\l \neq 1$ is an eigenvalue of the stochastic matrix $S$, then $|\l| \leq \r(S)$. 	 		
 		Our approach was based on the use of the following known optimization theorem, a form of
 		which was introduced in the study of the location of eigenvalues of real matrices,
 		for the first time (to the best of our knowledge), by Barany and Solymosi in their $2017$ paper \cite{BS}.
 		
 		\begin{theorem}\label{t3}
 			Consider the real function of the real variable
 			$f(x) = \displaystyle{\sum_{i=1}^n |x - \beta_i|}\, ,$ with $ \beta_1 \, \geq \,  ... \, \geq \, \beta_n, \ \ $ 
 			not necessarily distinct, $n$ real numbers. 
 			\begin{enumerate}
 				\item If $\,n\,$ is odd, then $ \ \ \displaystyle{ \min_{x\in \mathbb{R}}\, f(x) \ \ =
 					\ \ \big( \beta_1 + ... + \beta_{\frac{n-1}{2}} \big) \, - \, \big( \beta_{\frac{n+3}{2}} + ... + \beta_{n} \big)}. \ \ $ 
 				This minimum is reached when $ \ \ x \,  = \, \beta_{\frac{n+1}{2}} . $
 				\item If $\,n\,$ is even, then $\ \ \displaystyle{ \min_{x\in \mathbb{R}}\, f(x) \ \ = 
 					\ \ \big( \beta_1 + ... + \beta_{\frac{n}{2}} \big) \, - \, \big( \beta_{\frac{n}{2}+1} + ... + \beta_{n} \big)}. \ \ $
 				This takes place for every $ \ \ x \,  \in  \, [\beta_{\frac{n}{2}} \ \ , \ \ \beta_{\frac{n}{2}+1}] \ \  $ 
 				if $ \ \ \beta_{\frac{n}{2}} \, \not= \, \beta_{\frac{n}{2}+1}\ \ $ and only for $ \ \ x = \beta_{\frac{n}{2} \ \ }$ if
 				$ \ \ \beta_{\frac{n}{2}}  =  \beta_{\frac{n}{2}+1}\, $.
 			\end{enumerate}
 		\end{theorem}
 		
 		We were able to show the following results in \cite{HM4}.
 		
 		\begin{theorem}\label{b1}
 			Let $A = [a_{ij}]$ be an $n \times n$ real matrix, and assume that $\lambda$ is an eigenvalue of $A$ associated 
 			with a left eigenvector $v$ that is orthogonal to $e$. Then 
 			$$ \ \ |\lambda| \leq  \hat{\rho}(A). $$
 		\end{theorem}
 		
 		\begin{theorem}\label{b2}
 			Let $A$ be an $n \times n$ real matrix having  eigenvector $e$. Suppose that $\lambda$ is an eigenvalue of $A$ not associated with $e$, 
 			or is associated with $e$ but has geometric multiplicity greater than $1$.  
 			Then $ \ \ |\lambda| \, \leq \, \hat{\rho}(A)$.  In particular, if  $A$ is a stochastic matrix,
 			then $ \ \ |\lambda| \, \leq \, \hat{\rho}(A)$ for every eigenvalue $\lambda $ of $A$ different than $1$.
 		\end{theorem}
 		
 		Even more important is the following enhancement of Theorem \ref{b2}; this allows for better bounds on the eigenvalues, as was shown in \cite{HM4}.
 		
 		\begin{cor}\label{b3}
 			Let $A$ be an $n \times n$ real matrix having  eigenvector $e$. Suppose that $\lambda$ is an eigenvalue of $A$ not associated with $e$, 
 			or is associated with $e$ but has geometric multiplicity greater than $1$.     
 			Then for every $\, m \, \in \, \mathbb{N}$,
 			$$  |\lambda| \ \ \leq \ \ \sqrt[m]{\hat{\rho}(A^m)} \, .$$
 		\end{cor}
		 In the next section we state and prove Corollary \ref{cc2} which provides an effective way to bound and, even better, approximate the largest non-trivial eigenvalue in absolute value of e-matrices by the use of $\{\tau_p\}$. In the same section, we
         discuss how some of the important known properties of $\{\tau_p\}$ can be extended to all e-matrices and
         we prove a convergence theorem similar to Gelfand's formula.          
         In the third section, we derive lower bounds for the smallest non-trivial eigenvalue of nonsingular e-matrices by 
         taking advantage of the closure of $E_n$ under matrix inversion and using  $\{\tau_p\}$.
         In the last section, we use some of the results obtained in the previous sections to bound the spectral radius of Laplacian matrices as well as the algebraic connectivity of their associated graphs. 
         At last, we show that $\tau_{1}(L) \leq \tau_{\infty}(L)$ for every Laplacian matrix $L$ associated with a simple graph. 		
 	\end{section}	
 	 	
	\begin{section}{Generalization of $\{\tau_p\}$ and their properties to e-matrices}
	\begin{subsection}{Explicit forms of $\tau_1$ and $\tau_{\infty}$ over $E_n$}
		Extension of the explicit form of $\tau_1$  to all e-matrices is known and can be found, for example, in \cite[Theorem 6.6]{IS} and \cite[Section 2]{Se1}. Extension of the explicit form of $\tau_{\infty}$ given by Theorem \ref{tt4} is also needed in our analysis as well as for the numerical applications of the provided examples. This can be done easily as shown below.
		\begin{theorem}\label{t1}
			Let  $A = [a_{ij}]$ be an $\n$ e-matrix and let $\l_A$ be its trivial eigenvalue.
			Then
			\begin{equation}\label{F29}
			\tau_1(A) = \frac{1}{2} \underset{i,j}{\max}\sum_{k=1}^n |a_{ik} - a_{jk}|
			= \l_A - \underset{i,j}{\min}\sum_{k=1}^n \min\{a_{ik} , a_{jk}\},			
			\end{equation}
			and
			\begin{equation}\label{F30}
			\tau_{\infty}(A) = \r(A),
			\end{equation}
			where $\r(A)$ is given by Definition \ref{d1}.
		\end{theorem}
		\begin{proof}		
			The case where $A=0$ is trivial.	
			\begin{enumerate}			
				\item If $A$ is nonnegative and $A\neq0$, then $\l_A>0$ and
				$A' = \frac{1}{\l_A} A$ is stochastic.
				Hence, applications of (\ref{F26a}) and (\ref{F26b}) to $A'$ lead, respectively, to (\ref{F29}) and (\ref{F30}).
				\item If $A$ is not nonnegative, then let $a<0$ be the smallest element of $A$. Choose any positive number $\epsilon$ and construct the nonnegative matrix $B= A+(\epsilon-a)J$ which has strictly positive constant row-sum equal to $\l_A+n(\epsilon-a)$. Using the definition of $\tau_p$ given by (\ref{F22}), we can see easily that $\tau_p(B) = \tau_p(A)$ for $p\in \mathbb{N}\cup \{\infty\}$. Then we apply (\ref{F26a}) and (\ref{F26b}) to the stochastic matrix $\frac{B}{\l_A+n(\epsilon-a)}$ to obtain, respectively, (\ref{F29}) and (\ref{F30}).
			\end{enumerate}
		\end{proof}		
	\end{subsection}
	\begin{subsection}{The functions $\tau_p$ are semi-norms on $E_n$}
		\begin{definition}\label{d0}
			A matrix semi-norm $f$ on $E_n$ is a mapping from $E_n$ to $\mathbb{R}$ that satisfies the following conditions:\\
			For all matrices $A , B \in E_n$ and for all $\alpha\in \mathbb{R}$,
			\begin{enumerate}     
				\item $f(A) \geq 0$.
				\item If $A=0$ then $f(A) = 0$.
				\item $f(\alpha\,A) = |\alpha| f(A)$.
				\item $f(A+B) \leq f(A)+f(B)$.
				\item $f(AB) \leq f(A) f(B)$.
			\end{enumerate} 
		\end{definition}
		Obviously, the first three axioms of the above definition hold for all functions  $\tau_p$ applied to all real matrices.\\
		The fourth one also holds for all $n \times n$ real matrices because 
                $$||(A+B)^T\,x||_p \leq ||A^T\,x||_p + ||B^T\,x||_p \leq \tau_p(A) + \tau_p(B)$$ for every $x \in \mathbb{R}^n$ with $||x||_p =1$ and $x^T e =0$,
		so that  $\tau_p(A+B) \leq \tau_p(A) + \tau_p(B)$.\\
		To understand why the sub-multiplicative property (the fifth one) holds for all e-matrices, we reproduce the proof given for stochastic matrices 
                in \cite[Theorem 3.6]{IS}, as applied to $E_n$.\\
		There exists a vector $y \in \mathbb{R}^n$ such that $||y||_p = 1, ~ y^Te=0$  and $\tau_p(AB) = ||(AB)^Ty||_p$.
		If $||(AB)^T\,y||_p =0$, then the sub-multiplicative property is obviously satisfied.
		Otherwise, let $x = \displaystyle{\frac{A^Ty}{||A^Ty||_p}}$. 
		Then $||x||_p = 1$ and since $A$ is an e-matrix, we also have $x^Te=0$.
		It follows that 
		$$ \tau_p(AB) = ||B^T (A^Ty)||_p = ||B^Tx||_p\, ||A^Ty||_p \leq \tau_p(B)\,\tau_p(A).$$
		This reasoning assumes that at least $A$ is an e-matrix. But we don't need the same assumption for $B$.
		However, the sub-multiplicative property may not hold  in the general case where $A$ and $B$ are real.
		We can show that for $\tau_1$ and $\tau_{\infty}$ by the following counter-example.
		\begin{example}\label{e4}
			Let 
			$$ A \,= \, \begin{bmatrix}
				~~1    &    ~~2  &   ~~3     \\
				-3     &   -1    &   -2      \\
				~~1    &    ~~1  &   ~~1     \\
			\end{bmatrix}
			\text{~~~and~~~}
			B \,= \, \begin{bmatrix}
			 1     &    2    &   1     \\
			 1     &    1    &   1      \\
			 2     &    1    &   1     \\
			 \end{bmatrix}.
			$$
			Then
			$$AB \,= \, \begin{bmatrix}
			~~9    &    ~~7  &   ~~6     \\
			-8     &   -9    &   -6      \\
			~~4    &    ~~4  &   ~~3     \\
			\end{bmatrix}.
			$$
			In this example, the sub-multiplicative property is not satisfied by $\tau_1$  or by $\tau_{\infty}$ for the product $AB$ 
			since
			 $$ \tau_1(AB) = \frac{45}{2} > \tau_1(A)\,\tau_1(B) = 6.$$
			 $$\tau_{\infty}(AB) = 17 >  \tau_{\infty}(A)\,\tau_{\infty}(B) = 5.$$
			 The above values are obtained by the use of the explicit forms of $\tau_1$ and $\tau_{\infty}$
			 given in Theorem \ref{t1}.		
		\end{example} 
		From the above discussion we can see that the functions $\tau_p$ are vector norms on $M_n(\mathbb{R})$
			regarded as a vector space of dimension $n^2$ and they are  matrix semi-norms on $E_n$, the set of $\n$ e-matrices.
	
	\end{subsection}
	\begin{subsection}{Functions $\tau_p$ as bounds for the eigenvalues of e-matrices.}
		In this subsection, we state more general versions of some of the theorems cited in the introduction.
		We start with an update of Theorem \ref{b1}.		
		\begin{theorem}\label{tt11}
			Let $A$ be an $\n$ real matrix having eigenvalue $\l$ associated with a left eigenvector $v$ that is orthogonal to $e$.
			Then
			\begin{equation}\label{F3}
			|\l| \leq \tau_p(A),
			\end{equation}
			for every $p \in \mathbb{N} \cup {\infty}$.
		\end{theorem}
		\begin{proof}
			The proof of Theorem \ref{tt11} is similar to the one given for \cite[Theorem 3.1]{RT} by Rothblum and Tan, except that they assumed that $A$   
                        is non-negative and $\l$ is its subdominant eigenvalue. Their proof remains valid if $A$ is any real matrix and $\l$ is an eigenvalue of $A$ 
                        associated with a left eigenvector $v$ that is orthogonal to $e$.\\
			First, notice that the definition of $\tau_p$ given by (\ref{F22}), can be written
			\begin{equation}\label{F31}
			\tau_p(A) = \underset{\underset{\underset{u \neq 0}{u^T e =0}}{u\in \mathbb{R}^n}}{\max} \frac{\|A^T u\|_p}{\|u\|_p}.
			\end{equation}
			There exist $x, y \in \mathbb{R}^n$  such that $v = x+iy$.
			Since $e^T v =0$, the real vector  $\cos \theta\, x - \sin \theta \,y$ is orthogonal to $e$ for every $\theta \in \mathbb{R}$. 
                         Therefore by (\ref{F31}),
			 \begin{equation}\label{F32}
			 \|\cos \theta\, x - \sin \theta \,y\|_p \,\tau_p(A) \geq \|A^T (\cos \theta\, x - \sin \theta \,y)\|_p.
			 \end{equation}
			Consider the function
			\begin{equation}\label{F33}
			N_p(v) = \underset{\theta \in \mathbb{R}}{\max}~ \|  \cos\theta\,x -\,\sin\theta\,y\|_p,
			\end{equation}
			 which is, evidently, a vector norm on $\mathbb{C}^n$.
			We can scale the eigenvector $v$ such that $N_p(v) = 1$.
			Hence, 
			\begin{align*}
			|\l|  &= |\l| N_p(v)\\
				  &= N_p(v^* A)\\
				  &= N_p(A^T v) \\
				  &= \underset{\theta \in \mathbb{R}}{\max}~ \|  A^T (\cos\theta\,x -\,\sin\theta\,y)\|_p\\				  
				  &\leq \underset{\theta \in \mathbb{R}}{\max}~ \|\cos\theta\,x -\,\sin\theta\,y\|_p \tau_p(A), ~~\text{by (\ref{F32})}\\
				  &= \tau_p(A)\,  \underset{\theta \in \mathbb{R}}{\max}~ \|\cos\theta\,x -\,\sin\theta\,y\|_p\\
				  &=\tau_p(A)\,N_p(v), ~~\text{by (\ref{F33})}\\
				  &= \tau_p(A).
			\end{align*}		
			\end{proof}		
		Given that $A$ is an e-matrix, we discuss a sufficient condition for $\tau_p(A)$ to be an upper-bound for the absolute value of the trivial eigenvalue $\l_A$ of $A$.
		One may wonder why we do this since $\l_A$ is trivially equal to the constant row-sum of $A$ and therefore need not to be bounded or approximated. This is to have a better understanding of the action of the functions $\tau_p$ on e-matrices  and also to have flexibility and ease in the statement and application of some of our following results. 
		\begin{lemma}\label{ll3}
				Let $A$ be an $\n$ e-matrix with $n\geq 2$.
				\begin{enumerate}
					\item For every left eigenvector $w$  of $A$ associated with a non-trivial eigenvalue $\l$, we have
					$$w^Te=0.
					$$
					\item If $\l_A$, the trivial eigenvalue of $A$ is not simple, then there is a left eigenvector $v$ associated with $\l_A$ such that $v^T e =0$.
				\end{enumerate}
		\end{lemma}
	\begin{proof}
		For the first assertion of the lemma, see the proof of \cite[Lemma 2.2]{HM4}. For the second assertion, see the proof of \cite[Corollary 2.14]{HM6}.
	\end{proof}	
		Note that even in the case where $\lambda_A$ is not simple
		with \gm\, equal to $1$,
		the left eigenvector $v$ of $A$  associated with $\lambda_A$,
		which is unique up to scalar multiplication, 
		is orthogonal to $e$ by Lemma \ref{ll3}. This fact 
		can be illustrated by the following example.
		\begin{example}
			Let $$ A \,= \, \begin{bmatrix}
			~~1   &   1    \\
			-1   &   3    \\
			\end{bmatrix} \\.$$
			A Jordan decomposition of $A$ is
			$$ A \,= \, \begin{bmatrix}
			-1   &   2    \\
			-1   &   1    \\
			\end{bmatrix}  \, 
			\begin{bmatrix}
			2   &   1    \\
			0   &   2    \\
			\end{bmatrix}  \,
			\begin{bmatrix}
			1   &   -2    \\
			1   &   -1    \\
			\end{bmatrix}.
			$$
			From the above, we can see that $A$ has trivial eigenvalue
			$\lambda_A = 2$ with \am  equal to $2$ and \gm~ equal to $1$.
			We can see also that 
			$(-1 \,,\, 1) A = 2 (-1 \,,\, 1)$ and 
			$A (1 \,,\, 1)^T = 2 (1 \,,\, 1)^T$. Obviously,
			the vectors $(-1 \,,\, 1)^T$  and $e= (1 \,,\, 1)^T$ are orthogonal to each other.
			This is consistent with Lemma \ref{ll3}.                               
		\end{example}
		\begin{theorem}\label{tt10}
			Let $A$ be an $\n$ e-matrix.
			Suppose that the trivial eigenvalue $\l_A$ of $A$ is not simple. Then
			\begin{equation}\label{F1}
			|\l_A| \leq \tau_p(A)
			\end{equation}
			and more generally, for every $k \in \mathbb{N},$
			\begin{equation}\label{F2}
			|\l_A| \leq \sqrt[k]{\tau_p(A^k)}.
			\end{equation}
		\end{theorem}
		\begin{proof}
			The first inequality
			follows from Theorem \ref{tt11} and Lemma \ref{ll3}.
			For the second inequality,
			we use a reasoning similar to the one in the proof of Corollary \ref{b3}. (\cite[Corollary 2.6]{HM4}).
		\end{proof}
		
		Using Theorem \ref{tt11}, Lemma \ref{ll3} and Theorem \ref{tt10}, we obtain the following updated versions of Theorem \ref{tt2} and Corollary \ref{b3}.
		\begin{theorem} \label{tt9}
			Let $A$ be an e-matrix and suppose that $\l$ is an non-trivial eigenvalue of $A$. Then
			\begin{equation}\label{F4}
			|\l| \leq \tau_p(A).
			\end{equation}
			If the trivial eigenvalue $\l_A$ is not simple, then we also have
			\begin{equation}\label{F5}
			|\l_A| \leq \tau_p(A).
			\end{equation}			
		\end{theorem}
		\begin{cor}\label{cc2}
			Let $A$ be an $\n$ e-matrix. Suppose that $\l$ is a non-trivial eigenvalue of $A$.
			Then for every $k \in \mathbb{N},$
			\begin{equation}\label{F6}
			|\l| \leq \sqrt[k]{\tau_p(A^k)}.
			\end{equation}
			If the trivial eigenvalue $\l_A$ is not simple, then we also have
			\begin{equation}\label{F7}
			|\l_A| \leq \sqrt[k]{\tau_p(A^k)}.
			\end{equation}
		\end{cor}		
	\end{subsection}

	\begin{subsection}{Convergence formula satisfied by the functions $\tau_p $}
			We show a convergence formula satisfied by all functions $\tau_p $.
		    It is similar to Gelfand's formula \cite{G}, but differs from it in the sense
			that the limit is not always equal to the spectral radius of the matrix.
			Let's first recall the statement of Gelfand's formula.
		\begin{theorem}\label{tt8}(Gelfand's formula).
			Let $A$ be a square complex matrix and let $\rho(A)$ be 
			the spectral radius of $A$. Let $||\,.\,||$ be a matrix norm on the set of complex matrices.
			Then 
			\begin{equation}\label{f6}
			\rho(A) = \lim_{k \to \infty}~||A^k||^{1/k}.
			\end{equation}
		\end{theorem}  
			To prove our convergence formula we need, besides Gelfand's formula, the following three lemmas.
			\begin{lemma}\label{ll1}
				Let $A$ be an e-matrix and
				let $\| A \|_p$ be the matrix norm of $A$ induced by the corresponding $l_p$ vector norm.
				Then
				$$ \tau_p(A) \leq \| A^T \|_p.
				$$ 
			\end{lemma} 
			\begin{proof}
				Follows from the definitions of $\tau_p(A)$ and $\|A^T\|_p$. 
			\end{proof}		
		The second lemma needed in our analysis was obtained in \cite[Corollary 2.3]{HM6}
		and is a immediate consequence of a well-known theorem about matrix deflation due to Alfred Brauer \cite{B} .\\
			\begin{lemma}\label{ll2}
			Let $A$ be an $\n$ e-matrix with $n\geq2$.
			Let $\lambda_A, \lambda_2, \dots, \lambda_n$ be the eigenvalues of $A$,
			where $\l_A$ is the trivial eigenvalue of $A$.
			Then the matrix $B = A + \alpha J$ where $\alpha \in \mathbb{R}$, has eigenvalues
			$\lambda_B, \lambda_2, \dots, \lambda_n$
			with $\lambda_B = \lambda_A + n\alpha$ as the trivial eigenvalue of $B$.
		\end{lemma}       
		\begin{remark}
			Note that the listed eigenvalues of $A$ in the above lemma are not necessarily distinct.
			Thus, the lemma holds if $\l_A$ is equal to one of the other eigenvalues of $A$. The same applies to $B$.
		\end{remark}
		\begin{lemma}\label{ll4}
			Let $A$ be an $\n$ e-matrix and let $\alpha$ be a real number. Then 
			\begin{equation}\label{F8}
			\tau_p((A+\alpha\,J)^k) = \tau_p(A^k), ~~~k=1, 2, \dots.
			\end{equation}
		\end{lemma}    
		\begin{proof}
			First, note that
			\begin{equation}\label{F9}
			J = e\,e^T \text{~~and~~} J^k = n^{k-1}J \text{~~for~~} k=1, 2, \dots.
			\end{equation}
			Therefore, if $B$ is an e-matrix having $\l_B$ as trivial eigenvalue and $\beta$ is a real number, then
			\begin{equation}\label{F10}
			\tau_p(B+\beta\,J) = \underset{\underset{\underset{||x||_p =1}{x^T e =0}}{x\in \mathbb{R}^n}}{\max} ||(B^T + \beta\, e\,e^T)x||_p = \tau_p(B),
			\end{equation}
			\begin{equation}\label{F11}
			\tau_p(J B) = \underset{\underset{\underset{||x||_p =1}{x^T e =0}}{x\in \mathbb{R}^n}}{\max} ||B^T e e^T x||_p = 0
			\end{equation}
			and
			\begin{equation}\label{F12}
			\tau_p(B J) = \tau_p(\l_B J) = |\l_B| \tau_p(J)  = 0.
			\end{equation}
			Note also that $A^m$ is an e-matrix for every $m \in \mathbb{N}$ since $A$ itself is an e-matrix.\\		
			Then the assertion of the lemma follows by expanding the product $(A+\alpha\,J)^k$ and applying (\ref{F10}), (\ref{F11}), (\ref{F12}) and the definition of $\tau_p$ given by (\ref{F22}).			
		\end{proof}
	   \begin{theorem}\label{t5}
		Let $A$ be an $\n$ e-matrix and let $\lambda_A, \lambda_1, \lambda_2, \dots, \lambda_{n-1}$
		be the (not necessarily distinct) eigenvalues of $A$, where $|\lambda_1| \leq |\lambda_2| \leq \dots \leq |\lambda_{n-1}|$ and $\lambda_A$ is the trivial eigenvalue of $A$.
		Then 		
		\begin{equation}\label{f21}
		\lim_{k \to \infty}~ \big(\tau_p(A^k)\big)^{1/k}=|\lambda_{n-1}|.
		\end{equation}	
		In particular, if $A$ is stochastic and irreducible, then (\ref{f21}) provides
		a convergence formula for the absolute value of the subdominant eigenvalue of $A$.	  	       	 
	\end{theorem}              
	\begin{proof}	We have two cases.\\		
			\textbf{1) ~~~ $|\lambda_{n-1}| > |\lambda_A|$}.  The spectral radius of $A$ is equal to $|\lambda_{n-1}|$.
			We use Corollary \ref{cc2},  Gelfrand's formula and Lemma \ref{ll1} to obtain
			$$ |\lambda_{n-1}| \leq \lim_{k \to \infty} \big(\tau_p(A^k)\big)^{1/k}  \leq \lim_{k \to \infty} ||(A^T)^k||_p^{1/k} = |\lambda_{n-1}|,$$ which leads to
			$$ \lim_{k \to \infty} \big(\tau_p(A^k)\big)^{1/k} = |\lambda_{n-1}|.$$
			\textbf{2) ~~~ $|\lambda_{n-1}| \leq |\lambda_A|$}.        		 
			We construct the matrix $B = A-\lambda_A\,J$
			which, according to Lemma \ref{ll2}, has eigenvalues 
			$0, \lambda_1,  \lambda_2, \dots, \lambda_{n-1}.\,$
			The trivial eigenvalue of $B$ is $\l_B=0$ and its spectral radius is equal to $|\lambda_{n-1}|$.
			Then we go back to the first case to deduce that
			\begin{equation}\label{f13}
			\lim_{k \to \infty} \big(\tau_p(B^k)\big)^{1/k} = |\lambda_{n-1}|.
			\end{equation}
			Then by Lemma \ref{ll4}, $\tau_p(B^k)=\tau_p((A-\lambda_A\,J)^k) = \tau_p(A^k)$, so that
			(\ref{f13}) becomes
			$$\lim_{k \to \infty} \big(\tau_p(A^k)\big)^{1/k} = |\lambda_{n-1}|. $$ 		
	   		
	\end{proof}
	\begin{remark}
		Equation (\ref{f21}) is particularly significant if $|\l_A|$ is simple and equal to the spectral radius of $A$,
		because it provides a convergence  formula for the second largest eigenvalue of $A$ in absolute value.
		This applies, for example, to irreducible stochastic matrices as mentioned in the theorem itself.
		If $|\l_A| \leq |\l_{n-1}|$, then  the spectral radius of $A$ is equal to $|\l_{n-1}|$
		and  (\ref{f21}) does the same job as Gelfand's formula,
		except it uses a matrix semi-norm instead of a matrix norm.              
	\end{remark}
		Note that the sequence $\sqrt[k]{\tau_p(A^k)}$ converges to $|\l_{n-1}|$, and one may hope that this convergence is monotone non-increasing for every e-matrix $A$, so that every increase of $k$ leads to a bound that is better or equal to the previous one.
		However, this is not always the case as shown by 
		the following counter-example.
	\begin{example}\label{ee2}
		Consider the $4\times4$ circulant matrix
		$$A = \begin{bmatrix}
		4  & 1  & 2 & 3\\
		3  & 4  & 1 & 2\\
		2  & 3  & 4 & 1\\
		1  & 2  & 3 & 4\\
		\end{bmatrix}.
		$$
		We have $\tau_{\infty}(A) = 4, ~\sqrt{\tau_{\infty} (A^2)}\approx 2.83~ , ~\sqrt[3]{\tau_{\infty}(A^3)} \approx 3.17~$ and $~\sqrt[4]{\tau_{\infty}(A^4)} \approx 2.83$.
		Thus, the sequence $\sqrt[k]{\tau_{\infty}(A^k)}$ in not  non-increasing for this particular matrix.
	\end{example}
		However, there are sub-sequences of $\sqrt[k]{\tau_p(A^k)}$ that are indeed monotone non-increasing for every e-matrix $A$. This is a consequence of the sub-multiplicative property of $\{\tau_p(A)\}$ as shown in the proof of the following corollary. 
	\begin{cor}\label{c6}		
		The sequence $\sqrt[2^k]{\tau_p(A^{2^k})}$ is monotone non-increasing for every e-matrix $A$.		
	\end{cor}
	\begin{proof}
		The matrix $A^{2^k}$ is as well an e-matrix and we have
		$$
		\begin{array}{ccl}
		\tau_p(A^{2^{k+1}}) &=& \tau_p(A^{2^k} A^{2^k})\\
		&\leq& \big(\tau_p(A^{2^k})\big)^2, ~~~\text{by the submultiplicative property of~~} \tau_p. \\ 
		\end{array}
		$$
		It follows that,
		$$ \sqrt[2^{k+1}]{\tau_p(A^{2^{k+1}})} \leq \sqrt[2^{k+1}]{(\tau_p(A^{2^k}))^2} 
		= \sqrt[2^k]{\tau_p(A^{2^k})}.$$   			
	\end{proof}
		This monotonic decrease is important for its numerical applications. It can be used,  for example,
		to improve the upper bound on the largest non-trivial eigenvalue in absolute value of e-matrices.
	\begin{example}\label{e2}
		Let $$ A = \begin{bmatrix}
		1 &~~~0  & ~1\\
		2 & -1    & ~1\\
		0 &~~~1  & ~1\\
		\end{bmatrix}.
		$$
		The spectrum  of $A$ is equal to $\{2, -1, 0\}$.
		The second largest eigenvalue of $A$ in absolute value is $\l = -1$.
		Let's take some powers of $A$. For example, $$ A^3 = \begin{bmatrix}
		3 & ~1  & ~4\\
		4 & ~0    & ~4\\
		2 & ~2  & ~4\\
		\end{bmatrix}
		\text{~~~~and~~~~}
		A^{10} = \begin{bmatrix}
		341 & 171 & 512\\
		340 & 172 & 512\\
		342 & 170 & 512\\
		\end{bmatrix}.
		$$
		\begin{center}
			\begin{tabular}{|c|c|c|c|}\hline
				              $A^k$         	 		&   $A$   &  $A^3$  & $A^{10}$   \\ \hline
				$\tau_{\infty}(A^k)$                  	&   $2$   &   $2$   &  $2$       \\ \hline 
				$\sqrt[k]{\tau_{\infty}(A^k)}\approx$ 	&   $2$   &  $1.26$ & $1.07$     \\ \hline
								$\tau_1(A^k)$          	&   $2$   &   $2$   & $2$        \\ \hline
				$\sqrt[k]{\tau_1(A^k)} \approx$ 	&   $2$   &  $1.26$ & $1.07$     \\ \hline
			\end{tabular}
		\end{center}		
		Applications  of $\tau_1$ and $\tau_{\infty}$ to $A^{10}$ led to a better upper bound.
		Perhaps, better upper bounds are obtained by using powers of $A$ higher than $10$.
		Theorem \ref{t5} and Corollary \ref{c6} ensures the existence of higher powers of $A$ allowing for bounds close to  $|\l| =1$ as much as desired.		
	\end{example}
	\end{subsection}
	\begin{subsection}{A criterion for the spectral radius of an e-matrix to be simple}
		From the Perron-Frobenius theorem, we know that the spectral radius of an irreducible nonnegative matrix is simple. That is, it's algebraic multiplicity is equal to $1$. Here, we state  a criterion for  the trivial eigenvalue $\l_A$ of an e-matrix $A$ (not necessary nonnegative) to be simple. This result follows from Theorem \ref{tt9}.
	\begin{cor}\label{c5}
		Let $A=[a_{ij}]$ be an $\n$ e-matrix and let $\lambda_A$ be its trivial eigenvalue.
		Suppose that $|\lambda_A| > \tau_p(A)$ for some $p \in \mathbb{N}\cup \{0\}$. Then
		$\lambda_A$ is simple and it is the largest eigenvalue of $A$ in absolute value.
		Moreover, if $\l$ is any non-trivial eigenvalue of $A$, then
		$|\lambda_A|-|\l| \geq  |\lambda_A|-\tau_p(A)>0$.
	\end{cor}
	\begin{proof}
		Follows, immediately, from Theorem \ref{tt9}.
	\end{proof}
	\begin{example} Let
		$$A=
		\begin{bmatrix}
		~~5  & 3 & -1   &  ~~~2\\
		~~3  & 5 &  ~~3 & -2  \\
		~~3  & 3 & ~~3  &  ~~~0\\
		-2    & 5 & ~~2  &  ~~~4\\
		\end{bmatrix}
		$$
		We can see that $\lambda_A = 9$ and $\tau_{\infty}(A)=8$.
		According to Corollary \ref{c5}, $\lambda_A = 9$ is simple and it is the largest eigenvalue of $A$. The gap between  $\lambda_A = 9$ and the nearest  eigenvalue to it in absolute value is greater than or equal to $\lambda_A - \tau_{\infty}(A) = 1$.       	
		\end{example}
	
		At the end of this section, we pose two       
		open questions about  $\tau_1(A)$ and $\tau_{\infty}(A)$.
		\begin{openquestion}
		Let $A$ be an e-matrix.
		Find a sufficient condition or, if possible, a sufficient and necessary condition
		for $\tau_{\infty}(A^k)$ or $\tau_1(A^k)$ to be constant in terms of $k$, ie,
	\begin{equation}
		\tau_1(A) = \tau_1(A^2) = \tau_1(A^3) = \dots
	\end{equation}
	or
	\begin{equation}\label{f28}
	\tau_{\infty}(A) = \tau_{\infty}(A^2) = \tau_{\infty}(A^3) = \dots
	\end{equation}
	\end{openquestion}
	\begin{openquestion}
		Let $A$ be an e-matrix.
		Prove or disprove the following propositions:\\
		1)~~~If $\tau_1(A) = \tau_1(A^2)$, then $\tau_1(A) = \tau_1(A^k)$ for every $k \in \mathbb{N}$.\\
		2)~~~If $\tau_{\infty}(A) = \tau_{\infty}(A^2)$, then $\tau_{\infty}(A) = \tau_{\infty}(A^k)$ for every $k \in \mathbb{N}$.
	\end{openquestion}

	\begin{example}\label{e3}
		The matrix $$ A = \begin{bmatrix}
		1 &~~~0  & ~1\\
		2 & -1    & ~1\\
		0 &~~~1  & ~1\\
		\end{bmatrix},
		$$
		used in Example \ref{e2}, seems to satisfy (\ref{f28}) since $\tau_{\infty}(A) = \tau_{\infty}(A^2) = \dots = \tau_{\infty}(A^{10}) = 2$.\\
		Just a permutation of the elements of $A$ yields the e-matrix 
		$$ 
		B = \begin{bmatrix}
		1 &~~~1  & ~0\\
		2 & -1    & ~1\\
		0 &~~~1  & ~1\\
		\end{bmatrix},
		$$
		which does not have the same property as $A$ since $\tau_{\infty}(B) = 2$ and $\tau_{\infty}(B^2) = 4$.
	\end{example}   
	\end{subsection}     
	\end{section}
\begin{section}{Using $\{\tau_p\}$ to bound the smallest non-trivial eigenvalue in absolute value of e-matrices}
	\begin{subsection}{Non-singular e-matrices}
		Let $A$ be an $\n$ nonsingular e-matrix and let $B = A^{-1}$ be its inverse matrix.
		Suppose that the spectrum of $A$ consists of $\{\l_A, \l_1, \l_2, \dots, \l_{n-1}\}$,
		with $\l_A$ being the trivial eigenvalue of $A$ and
		\begin{equation}\label{F15}
		|\l_1| \leq |\l_2| \leq \dots \leq |\l_{n-1}|.
		\end{equation}
		Then $B$ itself is an e-matrix having spectrum $\Big\{ \frac{1}{\l_A}, \frac{1}{\l_1}, \frac{1}{\l_2}, \dots, \frac{1}{\l_{n-1}}\Big\}$, where $\l_B = \frac{1}{\l_A}$ is the trivial eigenvalue of $B$.  From (\ref{F15}), we have		 
		\begin{equation}\label{F16}
		\frac{1}{|\l_{n-1}|} \leq \frac{1}{|\l_{n-2}|} \leq \dots \leq \frac{1}{|\l_1|}.
		\end{equation}  
		By using (\ref{F16}) and applying Corollary \ref{cc2} to the e-matrix $B$, we obtain a lower bound for $\l_1$,  stated as follows.
		\begin{theorem}\label{t6}
			Let $A$ be an $\n$ nonsingular e-matrix having spectrum\\
			$\{\l_A, \l_1, \l_2, \dots, \l_{n-1}\}$,
			such that $\l_A$ is the trivial eigenvalue of $A$ and
			\begin{equation*}
			|\l_1| \leq |\l_2| \leq \dots \leq |\l_{n-1}|.
			\end{equation*}
			Then
			\begin{equation}\label{F17}
			|\l_1| \geq \frac{1}{\sqrt[k]{\tau_p(A^{-k})}},  ~~~ k = 1, 2, \dots.
			\end{equation}
		\end{theorem}
		\begin{rem}
			Even in the case where $|\l_1| = |\l_A|$, the theorem still valid. This is ensured by Corollary \ref{cc2}.
		\end{rem}
		By applying Theorem \ref{t5} to the inverse matrix of $A$, we obtain a convergence formula for $|\l_1|$.
		\begin{theorem}\label{t7}
			Let $A$ be an $\n$ nonsingular e-matrix having spectrum\\
			$\{\l_A, \l_1, \l_2, \dots, \l_{n-1}\}$,
			such that $\l_A$ is the trivial eigenvalue of $A$ and
			\begin{equation*}
			|\l_1| \leq |\l_2| \leq \dots \leq |\l_{n-1}|.
			\end{equation*}
			Then
			\begin{equation}\label{F18}
			\lim_{k \to \infty}~ \big(\tau_p(A^{-k})\big)^{1/k}=\frac{1}{|\lambda_1|}.
			\end{equation}
		\end{theorem}		
	\end{subsection}
	\begin{subsection}{Singular e-matrices with simple trivial eigenvalue equal to $0$}
		Let $A$ be an $\n$ singular e-matrix having spectrum $\{0, \l_1, \l_2, \dots, \l_{n-1}\}$ with $\l_A=0$ is the trivial eigenvalue of $A$ and 
		\begin{equation*}
		|\l_1| \leq |\l_2| \leq \dots \leq |\l_{n-1}|.
		\end{equation*}
		Suppose that $\l_A =0$ is simple.
		According to Lemma \ref{ll2}, for every nonzero real number $\a$, the e-matrix $B_{\a}=A+\alpha J$ is nonsingular with  spectrum
		$\{\alpha, \l_1, \l_2, \dots, \l_{n-1}\},$ where $\l_{B_{\a}}=\alpha$ is the trivial eigenvalue of $B_{\a}$.
		Then we use Theorem \ref{t6} and Theorem \ref{t7} to obtain:
		\begin{cor}\label{c7}
			Let $A$ be a $\n$ singular e-matrix having spectrum\\
			$\{0, \l_1, \l_2, \dots, \l_{n-1}\}$.
			Suppose that $\l_A =0$ is simple and
			\begin{equation*}
			 |\l_1| \leq |\l_2| \leq \dots \leq |\l_{n-1}|.
			\end{equation*}
			Then for every nonzero real $\a$
			\begin{equation}\label{F19}
			|\l_1| \geq \frac{1}{\sqrt[k]{\tau_p((A+\alpha J)^{-k})}},  ~~~ k = 1, 2, \dots.
			\end{equation}
			and
			\begin{equation}\label{F20}
			\lim_{k \to \infty}~ \big(\tau_p((A+\alpha J)^{-k})\big)^{1/k}=\frac{1}{|\lambda_1|}.
			\end{equation}
		\end{cor}	
		Applications and examples related to the above results are provided in the next section.
	\end{subsection}
\end{section}
\begin{section}{Applications of $\{\tau_p\}$ in spectral graph theory}	
	The Laplacian matrix $L$ of a simple graph $G$ is singular and symmetric. It is an e-matrix with trivial eigenvalue $\l_L=0$ and every non-trivial eigenvalue of $L$ is positive. Let $\{ 0 \leq \l_2 \leq \l_3 \leq \dots \leq \l_n\}$ be the spectrum of $L$. 
	Eigenvalues with particular importance are $\l_n$ and $\l_2$. The eigenvalue $\l_2$ is called the algebraic connectivity of the graph $G$ and its is strictly positive if and only if $G$ is connected. 
	By Corollary \ref{cc2} and Corollary \ref{c7},  both $\l_n$ and $\l_2$ can be bounded by using the functions $\tau_p$.
	\begin{subsection}{Bounding the largest eigenvalue of the Laplacian matrix}			
		Upper bounds for $\l_n$ are obtained by straightforward application of Corollary \ref{cc2},
		\begin{equation}\label{F14}
		\l_n \leq \sqrt[k]{\tau_p(L^k)}, ~~~ k=1, 2, \dots.
		\end{equation}
		\begin{example}\label{ee3}
			In \cite{HM4}, we considered the Laplacian matrix
			$$L  \, = \,  \begin{bmatrix}
			\ \ 3  &   -1   &   -1   &\ \ 0  & -1    &\ \ 0  &\ \ 0  \\
			-1      & \ \ 3  &\ \ 0   &    -1 &\ \ 0  &\ \ 0  &   -1  \\
			-1      &\ \ 0   &\ \ 4   &    -1 &\ \  0 & -1    &   -1  \\
			\ \ 0  &   -1   &   -1   &\ \ 4  & -1    & -1    &\ \ 0  \\
			-1      &\ \ 0   &\ \ 0   &  -1   &\ \ 3  & -1    &\ \ 0  \\
			\ \ 0  &\ \ 0   &   -1   &  -1   &   -1  &\ \ 4  &   -1  \\
			\ \ 0  &   -1   &   -1   &\ \ 0  &\ \ 0  & -1    &\ \ 3   \\
			\end{bmatrix}\\.
			$$
			There, upper-bounds were obtained for the largest eigenvalue $\l_7 \approx 6.21$ of $L$ by calculating $\r(L)$ and $\sqrt{\r(L^2)}$. Note that $\r(L) = \tau_{\infty}(L)$ by Theorem \ref{t1} ($\r$ is an explicit form of $\tau_{\infty}$).\\
			Now that we have in hand Corollary \ref{cc2}, we extend this example by calculating  both $\tau_{\infty}$ and $\tau_1$ up to $L^3$. 
			\begin{center}
				\begin{tabular}{|c|c|c|c|}\hline					
					$L^k$                       	 & $L$   & $L^2$ & $L^3$     \\ \hline 
					$\tau_1(L^k)$					 & $7$   & $46$  & $294$     \\ \hline
					$\sqrt[k]{\tau_1(L^k)}$          & $7$   & $6.78$  & $6.65$  \\ \hline
					$\tau_{\infty}(L^k)$			 & $7$   & $46$  & $294$     \\ \hline
					$\sqrt[k]{\tau_{\infty}(L^k)}$   & $7$   & $6.78$  & $6.65$  \\ \hline
				\end{tabular}
			\end{center}		
		\end{example}
	\end{subsection}
	\begin{subsection}{Bounding the algebraic connectivity}
		Let $G$ be a connected simple graph and let $L$ be is its $\n$ Laplacian matrix with spectrum
		$\{0 < \l_2 \leq \l_3 \leq \dots \leq \l_n\}$. In \cite[Example 3.10]{HM4}, we have shown how to use the generalized inverse of $L$ to give a lower bound for its algebraic connectivity $\l_2$. Here we  present two different techniques to generate lower bounds for $\l_2$  by Corollary \ref{cc2}.
		\begin{enumerate}

			\item \textbf{First method}. By using inequality (\ref{F19}) from Corollary \ref{c7}, we obtain
			\begin{equation}\label{F23}
			|\l_2| \geq \frac{1}{\sqrt[k]{\tau_p((L+\alpha J)^{-k})}}~,
			\end{equation}			
			where $~k=1, 2, \dots~$ and $~\a$ is any nonzero real number.		
			\begin{example}\label{e6}
				We use the same matrix $L$ as in \cite[Example 3.10]{HM4}.
				$$  L \,= \, \begin{bmatrix}
				\ \ 3  &   -1   &   -1     &  -1    \\
				-1    & \ \ 2  & \ \ 0    &  -1    \\
				-1    & \ \ 0  & \ \ 1    &\ \ 0    \\
				-1    &   -1   & \ \ 0    &\ \ 2    \\
				\end{bmatrix}.$$
				The spectrum of $L$ is $\{0, 1, 3, 4 \}$ and its algebraic connectivity is $\l_2=1$.
				We choose $\a=1$ so that
				$$  L+J \,= \, \begin{bmatrix}
				4    &  0  & 0    & 0    \\
				0    &  3  & 1    & 0    \\
				0    &  1  & 2    & 1    \\
				0    &  0  & 1    & 3    \\
				\end{bmatrix} \text{~~~and~~~} 
				(L+J)^{-1} \,= \, \begin{bmatrix}
				1/4    &  0     & 0       & 0       \\
				0      &  5/12  & -1/4    & 1/12    \\
				0      &  -1/4  & 3/4     & -1/4    \\
				0      &  1/12  & -1/4    & 5/12    \\
				\end{bmatrix}.
				$$				
				Letting $M = (L+J)^{-1}$,
				we have $\tau_1(M) =1$ ~and~ $1=\l_2 \geq \frac{1}{\tau_1(M)} = 1$.\\
				Using $\tau_{\infty}$, we have  $\tau_{\infty}(M) = 1.25$ ~and~ $1=\l_2 \geq \frac{1}{\tau_{\infty}(M)} = 0.8$.\\
				The lower bound $\frac{1}{\tau_1(M)} = 1$ is optimal for this particular matrix.
			\end{example}
			\item \textbf{Second method}. Let $\a \in \mathbb{R}^+$.
			Then $L+\a I$ is a nonsingular e-matrix with spectrum $\{\a, \l_2+\a, \l_3+\a, \dots, \l_n+\a\}$. By applying Corollary \ref{cc2} to its inverse matrix which is also an e-matrix, we get
			$$\frac{1}{\l_2+\a} \leq \sqrt[k]{\tau_p\big((L+\a I)^{-k}\big)},
			$$
			from which we obtain
			\begin{equation}\label{F25}
			\l_2 \geq \frac{1}{\sqrt[k]{\tau_p((L+\a I)^{-k})}} -\a, ~~~ k=1, 2, \dots,
			\end{equation}
			which implies that
			\begin{equation}\label{F34}
			\l_2 \geq \underset{\a \in \mathbb{R}^+}{\sup}\Big\{\frac{1}{\sqrt[k]{\tau_p((L+\a I)^{-k})}} -\a \Big\}, ~~~ k=1, 2, \dots,
			\end{equation}
			\begin{example}
				As an application of (\ref{F25}), we use the same Laplacian $L$ as in Example \ref{e6} and we choose $\a = 0.1$. Then
				$$  L + 0.1 I \,= \, \begin{bmatrix}
				3.1    &   -1   &   -1     &  -1    \\
				-1        & 2.1  & \ \ 0    &  -1    \\
				-1        & \ \ 0  & 1.1    &\ \ 0    \\
				-1        &   -1   & \ \ 0    &2.1    \\
				\end{bmatrix}
				$$
				and $(L + 0.1 I)^{-1}$ is given to the second decimal digit by
				$$  (L + 0.1 I)^{-1} \, \approx \, \begin{bmatrix}
				2.68    &  2.44   &   2.44    & 2.44    \\
				2.44    &  2.83   &   2.22    & 2.51   \\
				2.44    &  2.22   &   3.13    & 2.22    \\
				2.44    &  2.51   &   2.22    & 2.83    \\
				\end{bmatrix}.
				$$
				Using $\tau_1$, we have $\tau_1((L + 0.1 I)^{-1}) \approx 0.91$ ~and~ $$1=\l_2 \geq \frac{1}{\tau_1((L + 0.1 I)^{-1})}-0.1 \approx 1.$$
				Using $\tau_{\infty}$, we have  $\tau_{\infty}((L + 0.1 I)^{-1}) \approx 1.13$ ~and~ $$1=\l_2 \geq \frac{1}{\tau_{\infty}((L + 0.1 I	)^{-1})}-0.1 \approx 0.78.$$				
			\end{example} 
		\end{enumerate}	
	\end{subsection}
		\begin{subsection}{Comparison of $\tau_{1}$ and $\tau_{\infty}$ for Laplacian matrices}
		It is known that $\tau_{1} = \tau_{\infty}$ for all $2\times2$ and $3\times3$ stochastic matrices \cite{Se1}. This fact can be extended easily to all e-matrices. For higher dimensions, there are cases where $\tau_{\infty}< \tau_{1}$ such as the matrix $A$ in Example \ref{e4} and there are cases where $\tau_{1} < \tau_{\infty}$ such as the matrix $(L+J)^{-1}$ in Example \ref{e6}.
		For both of the $7\times7$ and $4\times4$ Laplacian matrices in, respectively, Example \ref{ee3} and Example \ref{e6}, we have $\tau_1 = \tau_{\infty}$. This equality is not a must for all Laplacian matrices as shown by the following example. 
		\begin{example}\label{e7} Let
			$$L  \, = \,  \begin{bmatrix}
			\ \ 3   &   -1   &   -1   & -1    &\ \ 0     &\ \ 0    \\
			-1      & \ \ 1  &\ \ 0   & \ \ 0 &\ \ 0  &\ \ 0    \\
			-1      &\ \ 0   &\ \ 1   & \ \ 0 &\ \  0 &\ \ 0    \\
			-1      &\ \ 0   &\ \ 0   &\ \ 2  & -1    &\ \ 0    \\
			\ \ 0   &\ \ 0   &\ \ 0   &  -1   &\ \ 2  &   -1    \\
			\ \ 0   &\ \ 0   &\ \ 0   &\ \ 0  & -1    &\ \ 1    \\
			\end{bmatrix}\\.
			$$
			The spectral radius of $L$ is $\l \approx 4.21 < \tau_1(L) =5 <\tau_{\infty}(L) =6$.
		\end{example}
		In \cite[Theorem 3.8]{HM4}, we have obtained the following result.
		Let $G$ be a simple graph on $n$ vertices and let $L =[l_{ij}]$ be its Laplacian matrix. Let $d$ be the largest degree among the vertices of $G$.
		Then
		\begin{equation}\label{F36}
		\r(L) = n, \text{~~if~~} d \geq \frac{n}{2} \text{~~~and~~~} \r(L) =2d, \text{~~if~~} d < \frac{n}{2}.
		\end{equation}
		By Theorem \ref{t1}, $\r(L)$ is the same as $\tau_{\infty}(L)$. Likewise, we can obtain a nice formula for $\tau_{1}(L)$ as follows.
		We use, for this purpose, the second equation of (\ref{F29}) according to which
		\begin{equation}\label{F37}
		\tau_{1}(L) = -\underset{i,j}{\min}\sum_{k=1}^n \min\{l_{ik},l_{jk}\}.
		\end{equation}
		Let $\{v_1, v_2, \dots, v_n\}$ be the set of vertices of $G$. For $i=1, 2, \dots, n$, let $d_i$ be the degree of the vertex $v_i$ and let $N_i$ be the set its neighbors. For fixed $i, j$ and $k$,
		observe that if any of $v_i$ and $v_j$ is connected to $v_k$, then $\min\{l_{ik},l_{jk}\}=-1$. Otherwise $\min\{l_{ik},l_{jk}\}=0$.
		It follows that
		$ \sum_{k=1}^n \min\{l_{ik},l_{jk}\} = -\big( d_i+d_j-|N_i \cap N_j| \big)$, where
		$|N_i \cap N_j|$ is the number of common neighbors of $v_i$ and $v_j$. Then we use (\ref{F37}) to obtain
		\begin{align*}
		\tau_1(L) =& -\underset{i,j}{\min}\Big\{-\big( d_i+d_j-|N_i \cap N_j| \big)\Big\}\\
		=&~~~~ \underset{i,j}{\max} \Big\{ d_i+d_j-|N_i \cap N_j| \Big\} .
		\end{align*} 			
		\begin{theorem}\label{t8}
			Let $G$ be a simple graph with vertices $\{v_1, v_2, \dots, v_n\}$ and let $L$ be its Laplacian matrix. For $i=1, 2, \dots, n$, let $d_i$ be the degree of the vertex  $v_i$ and let $N_i$ be the set of neighbors of the vertex $v_i$. Then
			\begin{equation}\label{F38}
			\tau_1(L) = \underset{i,j}{\max} \Big\{ d_i+d_j - |N_i \cap N_j|\Big\}.
			\end{equation}
		\end{theorem}
		Since $\l_n \leq \tau_1(L)$, it follows that
		\begin{equation}\label{F39}
		\l_n \leq \underset{i,j}{\max} \Big\{ d_i+d_j - |N_i \cap N_j|\Big\}.
		\end{equation}
		What is surprising is that the bound on $\l_n$ given by (\ref{F39}) was obtained in a $2000$ note \cite{RHR} by O. Rojo, R. Soto and H. Rojo without involving $\tau_1$ or anything  else about ergodicity coefficients. 
		In $2003$, this bound was improved by Kch Das as follows \cite[Theorem 2.1]{D}:
		\begin{equation}\label{F40}
		\l_n \leq \max \Big\{ d_i+d_j - |N_i \cap N_j|: v_iv_j \text{~is an edge in~} G\Big\}.
		\end{equation}
		Of course, all these bounds can by improved by applying Corollary \ref{cc2} to higher powers of $L$ as we did for the Laplacian matrix in Example \ref{ee3}.
		We believe that all the bounds existing in the literature are surpassed by $\sqrt{\tau_{1}(L^2)}$ or at most by $\sqrt[3]{\tau_1(L^3)}$ for the Laplacian matrix $L$ of every connected simple graph. This deserves further investigation. 
		\begin{example}
			Considering the same Laplacian matrix $L$ as in Example \ref{e7}, the spectral radius is $\l_6 \approx 4.21$.
			The bound given by (\ref{F40}) is the same as $\tau_1(L) = 5$. Applying Corollary \ref{cc2} to the  e-matrix $L^2$, we have
			$$\sqrt{\tau_1(L^2)} \approx 4.69.$$
		\end{example}
		
		To compare $\tau_1$ to $\tau_{\infty}$ for Laplacian matrices of simple graphs, we use their explicit forms given, respectively, by (\ref{F38}) and (\ref{F36}). This leads to the following result.
		\begin{theorem}\label{t10}
			Let $L$ be the Laplacian matrix of a simple graph. Then
			\begin{equation}
				\tau_1(L) \leq \tau_{\infty}(L).
			\end{equation}		
		\end{theorem}
		\begin{proof}
			Observe that $d_i+d_j - |N_i \cap N_j|\leq n$ for $i, j \in \{1, 2, \dots, n\}$.
			By (\ref{F38}), it follows that
			\begin{equation}\label{F41}
			\tau_1(L) \leq n.
			\end{equation}  					
			If $d\geq \frac{n}{2}$, then by (\ref{F36}) and (\ref{F41}), $\tau_1(L) \leq \tau_{\infty}(L)$.\\
			If $d < \frac{n}{2}$, then by (\ref{F36}), $\tau_{\infty}(L) = 2d$ so that $ \tau_{\infty}(L) \geq d_i+d_j - |N_i \cap N_j|$ for all $i, j \in \{1, 2, \dots, n\}$. Thus, $\tau_{\infty}(L) \geq \tau_1(L)$.
		\end{proof}	
		
		\begin{cor}\label{c8}
			\item Let $G$ be a connected simple graph with $n$ vertices and let $L$ be its Laplacian matrix.
			If $n\leq 5$, then $\tau_1(L) = \tau_{\infty}(L)$.		
		\end{cor}
		\begin{proof}
			The assertion of the corollary is easily checked case by case.
		\end{proof}
		\begin{rem}
			Corollary \ref{c8} does not apply to disconnected graphs. For example,		
			the following $4 \times 4$ Laplacian matrix $L$ is associated with a disconnected simple graph.
			$$  L \,= \, \begin{bmatrix}
			~~~2      &   -1   &   -1    &  0    \\
			-1        & ~~~2   &   -1    &  0    \\
			-1        &   -1   & ~~~2    &  0    \\
			~~~0      &  ~~~0  & ~~~0    &  0    \\
			\end{bmatrix}.
			$$
			We have $\tau_1(L) = 3 < \tau_{\infty}(L)=4$.
		\end{rem}
	\end{subsection}
	\end{section}
	\begin{section}{Conclusion}
		In this work,  $\tau_1$ and $\tau_{\infty}$ are used by the means of 
		Corollary \ref{cc2}  to give upper-bounds on the largest non-trivial eigenvalue in absolute value of e-matrices. However, Corollary \ref{cc2} and Corollary \ref{c6} allow for better than that. Theoretically speaking, they are algorithms to approximate the absolute value of this eigenvalue and Theorem \ref{t5} ensures their convergence. Questions about whether or not these algorithms are practical, cases where they may converge rapidly and how they can be improved may be the subject of further research.
	\end{section}	
	%
     
\end{document}